\documentclass[11pt]{amsart}
\usepackage{amsfonts,amssymb,amsmath}
\usepackage{latexsym}
\usepackage{color}
\usepackage{graphicx}
\usepackage{epsfig}
\usepackage{hyperref}

\usepackage[active]{srcltx}

\newtheorem{theorem}{Theorem}[section]
\newtheorem{lemma}[theorem]{Lemma}
\newtheorem{proposition}[theorem]{Proposition}
\newtheorem{corollary}[theorem]{Corollary}

\newtheorem{example}[theorem]{\sf Example}
\newtheorem{observation}[theorem]{Observation}

\newcommand {\C}    {{\mathbb C}}

\newcommand {\R}    {{\mathbb R}}
\newcommand {\N}    {{\mathbb N}}

\newcommand {\ag}    {\mathfrak{a}}
\newcommand {\bg}    {\mathfrak{b}}

\newcommand{\pa}{\partial}

\newcommand{\non}{\nonumber}

\begin{document}

\date{\today}

\title[Poly-analytic functions {\bf AND} representation theory]{Poly-analytic functions AND representation theory}
\author{Alexander {V} Turbiner}
\address{UNAM, Mexico-city, Mexico}
\email{turbiner@nucleares.unam.mx}
\author {Nikolai {L} Vasilevski}
\address{Cinvestav, Mexico-city, Mexico}
\email{nvasilev@cinvestav.mx}

\begin{abstract}
We propose the Lie-algebraic interpretation of poly-analytic functions in $L_2(\C,d\mu)$, with the Gaussian measure $d\mu$, based on a flag structure formed by the representation spaces of the $\mathfrak{sl}(2)$-algebra realized by differential operators in $z$ and $\bar z$. Following the pattern of the one-dimensional situation, we define poly-Fock spaces in $d$ complex variables in a Lie-algebraic way, 
as the invariant spaces for the action of generators of a certain Lie algebra. In addition to the basic case of the algebra $\mathfrak{sl}(d+1)$, we consider also the family of algebras $\mathfrak{sl}(m_1+1) \otimes  \ldots  \otimes \mathfrak{sl}(m_n+1)$ for tuples
$\mathbf{m} = (m_1,m_2,\ldots,m_n)$
of positive integers whose sum is equal to $d$.
\end{abstract}

\maketitle

\section{Introduction}

The paper deals with some aspects of the theory of the so-called poly-analytic functions (see e.g. \cite{Abreu,Balk,Hadelnmalm} and the literature cited therein). Recall that, in case of one complex variable, the poly-analytic functions of order $k$ are those that satisfy the iterated Cauchy-Riemann equation
\begin{equation*}
 \left(\frac{\partial}{\partial \overline{z}}\right)^k f = 0\ ,\ k \in \mathbb{N}.
\end{equation*}
One of the important and interesting questions here is as follows. {\it Given a domain $\Omega \subset \C^d$, $d \geq 1$, consider the corresponding (weighted) Hilbert space $L_2(\Omega,d\mu)$. How the poly-analytic functions are located inside $L_2(\Omega,d\mu)$, and whether and how the Hilbert space $L_2(\Omega,d\mu)$ can be made of the sets of poly-analytic function of different orders} (see e.g. \cite{Karlovich,Vasilev99,Vasilev00,Vasilev06}). This question is more challenging in the case of several complex variables, as by now there is no commonly accepted understanding how to define poly-analytic functions in this situation.

In the paper our Hilbert space is the standard $L_2(\C^d,d\mu)$, $d \geq 1$, with the Gaussian measure. The essence of our results is that the spaces of poly-analytic functions (both in one- and multidimensional cases) can be identified with the invariant subspaces under the action of certain Lie algebras.

In Section \ref{oneFock} we consider in detail the one-dimensional case of \cite[Section 2]{Vasilev00} and show that for each $k \in \mathbb{N}$ the poly-analytic space of order $k$, $k$-poly-Fock space $F^2_k(\C)$, can be alternatively defined as the common invariant subspace for the action the operators $J^+_k$, $J^0_k$, $J^-_k$ in $L_2(\C,d\mu)$,
\begin{eqnarray*}
   J^+_k &=&  \left(\overline{z}-\textstyle{\frac{\partial}{\partial z}}\right)\left(\overline{z} \textstyle{\frac{\partial}{\partial \overline{z}}} - \frac{\partial^2}{\partial z \partial \overline{z}}- (k-1)I\right),
\\
   J^0_k &=&  \overline{z} \textstyle{\frac{\partial}{\partial \overline{z}}} - \textstyle{\frac{\partial^2}{\partial z \partial \overline{z}}}- \textstyle{\frac{k-1}{2}}I,
\\
   J^-_k &=& \textstyle{\frac{\partial}{\partial \overline{z}}},
   \end{eqnarray*}
cf. \cite[Formulas (A.1.5)]{t-sl22}, obeying the $\mathfrak{sl}(2)$-algebra commutation relations.

It turns out that the situation of Section \ref{oneFock} is in fact a particular case of a general construction based on the Fock space formalism extended to the case of infinite dimensional space generated by vacuum vectors. These results are presented in Section \ref{se:overture}.

In final Section \ref{multiFock} we define poly-Fock spaces in $d$ complex variables as the representation spaces for $\mathfrak{sl}(d+1)$ (Subsection \ref{homog}), and for $\mathfrak{sl}(m_1+1) \otimes  \ldots  \otimes \mathfrak{sl}(m_n+1)$ for a given tuple
$\mathbf{m} = (m_1,m_2,\ldots,m_n)$
of natural numbers whose sum is equal to $d$, $m_1 +m_2+\ldots +m_n=d$ (Subsection \ref{quasi-homog}).

\section{ Fock space formalism}
\label{se:overture}

Introduce the three-dimensional Heisenberg algebra $\mathbb{H}_3 = \{ \ag,\bg,1 \}$ with commutator $[\ag, \bg]\ = \ 1$ and $[\ag, 1]=[\bg, 1]=0$. \\
A formal construction of the Fock space, see e.g. \cite[Chapter 5, Section 5.2]{BerezinShubin},  is given by the representation of the algebra $\mathbb{H}_3$ on a separable Hilbert space $\mathcal{H}$. That is, there are operators $\ag$ and $\bg$ acting in $\mathcal{H}$ that satisfy the above commutation relations, and there also is a normalized element $|0\rangle :=\Phi_0 \in \mathcal{H}$, $\|\Phi_0\|=1$, called the vacuum vector, such that $\ag \Phi_0 =0$ and the linear span of elements $\bg^n\Phi_0$ with $n \in \mathbb{Z}_+$ is dense in~$\mathcal{H}$. The Hilbert space $\mathcal{H}$ obeying the above properties can be called the \emph{Fock space related to  $\ag$ and $\bg$}, or \emph{$(\ag,\bg)$-Fock space}.

The problems of complex analysis related to the study of the so-called poly-analytic functions, i.e., those smooth functions that satisfy the equation
$$\left(\frac{\partial}{\partial \overline{z}}\right)^k f = 0\ ,\ k \in \mathbb{N}\ ,$$
motivate us to deal with a more general situation: There is a space $L_1$ spanned by vacuum vectors, i.e., $\ag h = 0$, for all $h \in L_1$, and the linear span of all elements from all spaces $L_n := \bg^{n-1} L_1$, $n \in \mathbb{N}$ is dense in $\mathcal{H}$.  Note that generically the elements from different spaces $L_n$ are not orthogonal.
The previous situation corresponds to the case when $L_1$ is one-dimensional and is generated by $\Phi_0$.

Again, motivated by the study of the poly-analytic functions, we define the $k$-poly-$\ag$-space $\mathcal{H}_k$ as the closure of all elements of $\mathcal{H}$ that satisfy the equation
\begin{equation} \label{a^k}
 \ag^k h = 0\ .
\end{equation}
Easily verified relation
\begin{equation} \label{[a,b]}
 [\ag,\bg^{n-1}] = (n-1)\bg^{n-2}
\end{equation}
implies that for each $h_1 \in L_1$ the element $h_n\  \equiv\ \bg^{n-1}h_1$ belongs to $L_n$. It follows from
\begin{eqnarray*}
 \ag^k h_n &=& \ag^k\bg^{n-1}h_1 = \ag^{k-1} \left(\ag\bg^{n-1}\right)h_1 = \ag^{k-1}\left(\bg^{n-1}\ag + (n-1)\bg^{n-2}\right)h_1 \\
 &=&  (n-1)\ag^{k-1}\bg^{n-2}h_1 = \ldots = (n-1)\cdots(n-(k-1))\ag \bg^{n-k}h_1
\end{eqnarray*}
for $k=n$.
Thus $\mathcal{H}_k = \mathrm{closure}(L_1 + L_2 + \ldots + L_k)$. Note that \eqref{a^k} is equivalent to
\begin{equation}
\label{bk-ak}
    \bg^k\ag^k h\ =\ 0\ .
\end{equation}
By induction on $k$ one can show, see \cite{Turbiner:1994}, that
\begin{equation*}
 \bg^k\ag^k\, = \, \prod_{m=0}^{k-1}\left(\bg \ag - m\right),
\end{equation*}
where the right hand side is a product of the so-called Euler-Cartan operators
\[
   I_0^{(m)}\ =\  \bg \ag - m,     \quad m=0,1,\ldots k-1\ .
\]
This reveals a Lie-algebraic nature of the $k$-poly-$\ag$-space $\mathcal{H}_k$.
Indeed, consider three (generically unbounded densely defined in $\mathcal{H}$) operators, see
\cite[Formulas (A.1.5)]{t-sl22}
\begin{eqnarray} \label{sl2ag}
   J^+_k &=&  \bg^2 \ag - (k-1)\, \bg , \non \\
   J^0_k &=&  \bg \ag - \textstyle{\frac{k-1}{2}} , \\
   J^-_k &=& \ag \non .
\end{eqnarray}
For all $k \in \C$ these operators obey the $\mathfrak{sl}(2)$-algebra commutation relations $[J^-_k, J^+_k]=2 J^0_k$ and $[J^{\pm}_k, J^0_k]= \mp J^{\pm}_k$.
Restricting $k$ to positive integers, it is easily seen that the space
\begin{equation*}
 V_k := L_1 + L_2 + \ldots + L_k
\end{equation*}
is invariant under the action of the operators $J^+_k$, $J^0_k$ and $J^-_k$.

So the $k$-poly-$\ag$-spaces $\mathcal{H}_k$ can be defined alternatively as the closure of the invariant subspaces $V_k$ for the action the operators $J^+_k$,
$J^0_k$, $J^-_k$ in $\mathcal{H}$, obeying the $\mathfrak{sl}(2)$-algebra commutation relations.

\medskip

The situation becomes much more transparent and substantial in case when the operators $\ag$, $\bg$ are identified with the lowering and raising operators $\ag$ and $\ag^\dag$, i.e., when additionally the operator $\ag^\dag$ is adjoint to
$\ag$.

Before we proceed with its description let us give three examples, starting with a very classical one.

\begin{example}{\rm (See e.g. \cite[Chapter 5, Section 5.2]{BerezinShubin}) \
 Take $\mathcal{H} = L_2(\mathbb{R})$, choose $\ag = \frac{1}{\sqrt{2}}(x + \frac{d}{dx})$, $\ag^\dag = \frac{1}{\sqrt{2}}(x - \frac{d}{dx})$ as the lowering and raising operators, and the vacuum vector $\Phi_0 = (\pi)^{1/4} e^{-x^2/2}$ being the ground state of the harmonic oscillator.  The linear span of functions $(\ag^\dag)^n \Phi_0\,=\,c_n\,H_n (x)\, e^{-x^2/2}$, $n \in \mathbb{Z}_+$, with an appropriate constant $c_n$, is dense in $L_2(\mathbb{R})$, and the $k$-poly-$\ag$-space $\mathcal{H}_k$ coincides with the finite dimensional space of weighted Hermite polynomials $H_p (x)$,
\begin{equation*}
  e^{-x^2/2} \sum_{p=0}^{k-1} a_p H_p (x)\ .
\end{equation*} }
\end{example}

\begin{example} {\rm (See e.g. \cite[Chapter 5, Section 5.2]{BerezinShubin}) \
 Take $\mathcal{H} = F_2(\mathbb{C})$ being the space of all anti-analytic functions $f(\overline{z})$ endowed with the scalar product
 \begin{equation*}
  \langle f, g \rangle = \frac{1}{\pi} \int_{\mathbb{C}} f(\overline{z})\overline{g(\overline{z})} e^{-|z|^2} dxdy, \quad \overline{z}=x-iy.
 \end{equation*}
Then $\ag = \frac{\partial}{\partial \overline{z}}$, $\ag^\dag = \overline{z}$, and $\Phi_0 =1$. The $k$-poly-$\ag$-spaces $\mathcal{H}_k$ coincides with the finite dimensional space of  polynomials on $\overline{z}$ of degree not greater that $k-1$.
}
\end{example}

\begin{example} \ {\rm
Take $\mathcal{H} =L_{2}(\C,d\mu)$ of square-integrable functions
on $\C$ with the Gaussian measure
\begin{equation*}
   d\mu(z)=\pi^{-1}\,e^{-z\cdot {\bar z}}dv(z)\ ,
\end{equation*}
Then \cite[Formula (2.4)]{Vasilev00}, take $\ag = \frac{\partial}{\partial \overline{z}}$, $\ag^\dag = -\frac{\partial}{\partial z} + \overline{z}$, while the role of the vacuum vector $\Phi_0$ can be played by any normalized analytic function from $L_{2}(\C,d\mu)$. This example is considered in more detail in the next section.
}
\end{example}

So we are now in the following setup. There are a separable Hilbert space $\mathcal{H}$, two mutually adjoint lowering and raising operators $\ag$ and $\ag^\dag$ having common domain dense in $\mathcal{H}$, and a linear subspace $L_1$ being the kernel of $\ag$. Furthermore, the operators $\ag$ and $\ag^\dag$ obey the relation
\begin{equation} 
\label{commutator}
 [\ag, \ag^\dag]\ = \ I,
\end{equation}
and the union of all spaces $L_n := (\ag^\dag)^{n-1} L_1$ is dense in $\mathcal{H}$.

Then the $k$-poly-$\ag$-space $\mathcal{H}_k$, defined through equation \eqref{a^k}, coincides with $\mathrm{closure}(L_1 + L_2 + \ldots + L_k)$, where $L_n = (\ag^\dag)^{n-1}L_1$. A simple application of the relation $[\ag,(\ag^\dag)^n] = n(\ag^\dag)^{n-1}$ yields the following proposition.
\begin{proposition} \label{a_a+}
Different subspaces $L_n$ and $L_m$ are orthogonal.\\
For each $k = 2,3, ...$, the raising operator
\begin{equation*}
  \textstyle{\frac{1}{\sqrt{k-1}}}\,\ag^\dag|_{\overline{L}_{k-1}} : \overline{L}_{k-1} \ \longrightarrow \ \overline{L}_{k}
\end{equation*}
is an isometric isomorphism, and the lowering operator
\begin{equation*}
   \textstyle{\frac{1}{\sqrt{k-1}}}\,\ag|_{\overline{L}_{k}} : \overline{L}_{k} \ \longrightarrow \ \overline{L}_{k-1} ,
\end{equation*}
is its inverse. Here $\overline{L}_{n}:= \mathrm{closure}(L_n)$.
\end{proposition}

\begin{proof}
 Take $h_n = (\ag^\dag)^{n-1}h_1 \in L_n$, $g_m = (\ag^\dag)^{m-1}g_1 \in L_m$, and let $n > m$. Then
 \begin{eqnarray*}
  \langle h_n, g_m \rangle &=& \langle (\ag^\dag)^{n-1}h_1, (\ag^\dag)^{m-1}g_1 \rangle =
  \langle (\ag^\dag)^{n-m-1}h_1, \ag^{m}(\ag^\dag)^{m-1}g_1 \rangle \\
  &=& \langle (\ag^\dag)^{n-m-1}h_1, (m-1)!\,\ag\, g_1\rangle\ =\ 0\ .
 \end{eqnarray*}
Take now $h_{k-1} = (\ag^\dag)^{k-2}h_1 \in L_{k-1}$ and calculate
\begin{eqnarray*}
 \left\langle \textstyle{\frac{1}{\sqrt{k-1}}} \ag^\dag h_{k-1}, \textstyle{\frac{1}{\sqrt{k-1}}} \ag^\dag h_{k-1} \right\rangle &=& \frac{1}{k-1} \langle (\ag^\dag)^{k-1}h_1, (\ag^\dag)^{k-1}h_1 \rangle = \frac{1}{k-1} \langle (\ag^\dag)^{k-2}h_1, \ag(\ag^\dag)^{k-1}h_1 \rangle \\
 &=& \frac{1}{k-1} \langle (\ag^\dag)^{k-2}h_1, (k-1)(\ag^\dag)^{k-2}h_1 \rangle\ =\ \langle h_{k-1}, h_{k-1} \rangle\ .
\end{eqnarray*}
That is, the operator $\textstyle{\frac{1}{\sqrt{k-1}}}\,\ag^\dag$ is an isometric isomorphism of $L_{k-1}$ onto $L_k$, which extends by continuity onto the closure of these spaces.

Finally, take $h_{k} = (\ag^\dag)^{k-1}h_1 \in L_{k}$ and calculate
\begin{eqnarray*}
 \left\langle \textstyle{\frac{1}{\sqrt{k-1}}} \ag h_{k}, \textstyle{\frac{1}{\sqrt{k-1}}} \ag h_{k} \right\rangle &=& \frac{1}{k-1} \langle \ag (\ag^\dag)^{k-1}h_1, \ag(\ag^\dag)^{k-1}h_1 \rangle = \frac{1}{k-1} \langle (\ag^\dag)^{k-1}h_1, \ag^\dag \ag(\ag^\dag)^{k-1}h_1 \rangle \\
 &=& \frac{1}{k-1} \langle (\ag^\dag)^{k-1}h_1, (k-1)(\ag^\dag)^{k-1}h_1 \rangle \ =\ \langle h_{k}, h_{k} \rangle\ ,
\end{eqnarray*}
and the result follows.
\end{proof}

The proposition implies that the $k$-poly-$\ag$-space $\mathcal{H}_k$ admits the representation
\begin{equation*}
 \mathcal{H}_k\ =\ \overline{L}_{k}  \oplus \overline{L}_{k-1} \oplus \ldots \oplus \overline{L}_{1} = \overline{L}_{k}  \oplus \mathcal{H}_{k-1}\ .
\end{equation*}
Following the notion introduced in the theory of the poly-analytic spaces \cite{Vasilev99,Vasilev00}, we define the true-$k$-poly-$\ag$-space $\mathcal{H}_{(k)}$ as
\begin{equation*}
 \mathcal{H}_{(k)} := \mathcal{H}_k \ominus \mathcal{H}_{k-1} = (\ag^{\dag})^{k-1}\mathrm{closure}(L_1) \quad \mathrm{and} \quad \mathcal{H}_{(1)}:= \mathcal{H}_1\ =\ \mathrm{closure}(L_1)\ ,
\end{equation*}
in other words
\begin{equation}
\label{decomH_k}
 \mathcal{H}_k = \bigoplus_{\ell=1}^k \mathcal{H}_{(\ell)} \quad \mathrm{with} \quad \mathcal{H}_{(\ell)} = (\ag^{\dag})^{\ell-1} \mathcal{H}_1 \ .
\end{equation}
Iterating the statements of Proposition \ref{a_a+} we come to the following

\begin{corollary} \label{A_k}
 And for each $k \in \N$, the operator
\begin{equation} \label{1->k}
  \mathbf{A}_{(k)} := \textstyle{\frac{1}{\sqrt{(k-1)!}}}\, (\ag^\dag)^{\,k-1}|_{\mathcal{H}_1}  : \mathcal{H}_1 \ \longrightarrow \ \mathcal{H}_{(k)} ,
\end{equation}
gives an isometric isomorphism between $\mathcal{H}_1$ and the true-$k$-poly-$\ag$-space $\mathcal{H}_{(k)}$,
and the operator
\begin{equation} \label{k->1}
  \mathbf{A}_{(k)}^{-1} := \textstyle{\frac{1}{\sqrt{(k-1)!}}}\,{\ag}^{k-1}|_{\mathcal{H}_{(k)}} : \mathcal{H}_{(k)}\ \longrightarrow  \mathcal{H}_1 ,
\end{equation}
is an inverse isomorphism.
\end{corollary}

Thus, the true-$\ell$-poly-$\ag$-spaces are nothing but the  components of the orthogonal decomposition \eqref{decomH_k} of the  $k$-poly-$\ag$-spaces $\mathcal{H}_{k}$, each component is isomorphic to the closure of the space of vacuum vectors and is obtained by the ``lifting up'' this closure by the normalised powers of the raising operator $\ag^\dag$.
Furthermore, the density in $\mathcal{H}$ of the union of all spaces $L_k$ implies
\begin{proposition}
\label{prop:+}
 The following direct sum decomposition of $\mathcal{H}$ holds
 \begin{equation*}
  \mathcal{H}\ =\ \bigoplus_{k=1}^{\infty} \mathcal{H}_{(k)}\ .
 \end{equation*}
\end{proposition}
\noindent
In the classical case of one-dimensional space $L_1$, generated by a vacuum vector $\Phi_0$, the statement of proposition is nothing but the well known fact that the infinite system of elements $\frac{1}{\sqrt{(k-1)!}}\,(\ag^\dag)^{\,k-1}\Phi_0$, $k \in \mathbb{N}$ forms an orthonormal basis on $\mathcal{H}$.

The isomorphisms of Corollary \ref{A_k} imply the following isomorphisms
\begin{eqnarray} \label{isoH_k} \non
 \mathcal{H} &\cong& \ell_2(\mathbb{N}, \mathcal{H}_1) = \ell_2(\mathbb{N},\mathbb{R}) \otimes \mathcal{H}_1, \\
 \mathcal{H}_k &\cong& \underbrace{\mathcal{H}_1 \oplus \mathcal{H}_1\ldots  \oplus
  \mathcal{H}_1}_{k \  \mathrm{times}} = \mathbb{R}^k \otimes \mathcal{H}_1.
\end{eqnarray}
An equivalent to Proposition \ref{prop:+} statement can be formulated in terms of the $k$-poly-$\ag$-spaces $\mathcal{H}_{k}$ as follows.
\begin{corollary} \label{co:U_H}
The set of $k$-poly-$\ag$-subspaces $\mathcal{H}_{k}$, $k\in \N$, of the space $\mathcal{H}$ forms an infinite flag in~$\mathcal{H}$
\begin{equation} \label{flagH}
  \mathcal{H}_{1} \ \subset \  \mathcal{H}_{2} \ \subset \ ... \ \subset \  \mathcal{H}_{k} \subset \ ...  \subset \ \mathcal{H},
\end{equation}
and
\begin{equation*}
   \mathcal{H} = \overline{\bigcup_{k=1}^{\infty} \mathcal{H}_{k}}.
\end{equation*}
\end{corollary}

The generators \eqref{sl2ag} of the Lie algebra $\mathfrak{sl}(2)$ in the case (\ref{commutator}) have the form \cite{t-sl22}
\begin{eqnarray}
\label{sl2agH}
   J^+_k &\ =\ &  (\ag^\dag)^2 \ag - (k-1)\, \ag^\dag \ , \non \\
   J^0_k &\ =\ &  \ag^\dag \ag - \textstyle{\frac{k-1}{2}}I \ , \\
   J^-_k &\ =\ & \ag \non ,
\end{eqnarray}
 and they act on elements $h_{(\ell)} \in \mathcal{H}_{(\ell)}$ as follows
\begin{equation*}
 J^+_k\,h_{(\ell)} =
 \begin{cases}
  \ag^\dag (\ell-k)h_{(\ell)} \in \mathcal{H}_{(\ell+1)},  & \ell \neq k \\
  0,  & \ell = k
 \end{cases}, \ \
 J^0_k\,h_{(\ell)} = \textstyle{\frac{2\ell-k+1}{2}}h_{(\ell)}, \ \
J^-_k\,h_{(\ell)} =
 \begin{cases}
 \ag\, h_{(\ell)} \in \mathcal{H}_{(\ell-1)}, & \ell > 1 \\
 0, & \ell = 1
 \end{cases},
\end{equation*}
which implies that for positive integer $k$ the $k$-poly-$\ag$-space $\mathcal{H}_{k}$
is a common invariant subspace for the action of $J^+_k$, $J^0_k$, and  $J^-_k$, where they act boundedly.

We describe now the isomorphic isomorphism \eqref{isoH_k} in more detail.
\begin{observation} \label{matrix}
 {\rm
 By \eqref{decomH_k} each element $h_k \in \mathcal{H}_{k}$ can be represented as a sum of its mutually orthogonal components $h_k = h_{(1)} + h_{(2)} + \ldots + h_{(k)}$ from different true-poly-$\ag$-spaces $\mathcal{H}_{\ell}$, and which we will thus identify with a column vector $h_k = (h_{(1)},h_{(2)},\ldots,h_{(k)})^t$. Each element $g \in \mathbb{R}^k \otimes \mathcal{H}_1$ we also represent as a column vector $g = (g_1,g_2,\ldots,g_k)^t$. Then the isomorphism $\mathbf{A}_k :\, \mathbb{R}^k \otimes \mathcal{H}_1 \rightarrow \mathcal{H}_{k}$ is given by $h_k = \mathbf{A}_k g$,  in matrix form
 \begin{equation*}
\begin{bmatrix}
 h_{(1)} \\
 h_{(2)} \\
 \vdots \\
 h_{(k)} \\
\end{bmatrix}
=
\begin{bmatrix}
\mathbf{A}_{(1)} & 0 & \cdots & 0 \\
0 & \mathbf{A}_{(2)} & \cdots & 0 \\
\vdots  & \vdots  & \ddots & \vdots  \\
0 & 0 & \cdots & \mathbf{A}_{(k)}
\end{bmatrix}
\,
\begin{bmatrix}
 g_1 \\
 g_2 \\
 \vdots \\
 g_k \\
\end{bmatrix} ,
\end{equation*}
 the inverse isomorphism  $\mathbf{A}_k^{-1} :\, \mathbb{R}^k \otimes \mathcal{H}_1 \rightarrow  \mathbf{A}_k$ is given by
 \begin{equation*}
  \begin{bmatrix}
 g_1 \\
 g_2 \\
 \vdots \\
 g_k \\
\end{bmatrix}
=
\begin{bmatrix}
\mathbf{A}_{(1)}^{-1} & 0 & \cdots & 0 \\
0 & \mathbf{A}_{(2)}^{-1} & \cdots & 0 \\
\vdots  & \vdots  & \ddots & \vdots  \\
0 & 0 & \cdots & \mathbf{A}_{(k)}^{-1}
\end{bmatrix}
\,
\begin{bmatrix}
 h_{(1)} \\
 h_{(2)} \\
 \vdots \\
 h_{(k)} \\
\end{bmatrix} ,
\end{equation*}
where the operators $\mathbf{A}_{(\ell)}$ and $\mathbf{A}_{(\ell)}^{-1}$ are given by
\eqref{1->k} and \eqref{k->1} respectively. }
\end{observation}

\begin{observation} \label{max-inv} {\rm
We note that unless the dimension of the vacuum vector $\mathcal{H}_1$ space is one, the action of the operators \eqref{sl2agH} on $\mathcal{H}_k$ is reducible. The general form of all closed invariant subspaces for this action is as follows. Take any closed subspace $\mathcal{\hat H}$ of $\mathcal{H}_1$, then the space $\mathbf{A}_k (\mathbb{R}^k \otimes \mathcal{\hat H}) \subset \mathcal{H}_k$ is invariant for the action of \eqref{sl2agH}. The space $\mathcal{H}_k$ is maximal (under inclusion) among all those invariant subspaces.}
\end{observation}

Given a Hilbert space $H$, we denote by $\mathcal{L}(H)$ the set of all bounded linear operators acting on~$H$. \\
Observation \ref{matrix} implies that there is one-to-one correspondence between elements of $\mathcal{L}(\mathcal{H}_{k})$ and elements of
\begin{equation*}
 \mathcal{L}(\mathbb{R}^k \otimes \mathcal{H}_1) = \mathcal{L}(\mathbb{R}^k) \otimes \mathcal{L}(\mathcal{H}_1) = \mathrm{Mat}_k(\mathbb{R}) \otimes \mathcal{L}(\mathcal{H}_1) = \mathrm{Mat}_k(\mathcal{L}(\mathcal{H}_1)),
\end{equation*}
i.e., each element $T \in \mathcal{L}(\mathcal{H}_{k})$ is unitary equivalent to a certain element $\widetilde{T} \in \mathcal{L}(\mathbb{R}^k \otimes \mathcal{H}_1)$, which are related by $T = \mathbf{A}_k \widetilde{T} \mathbf{A}_k^{-1}$ or $\widetilde{T} = \mathbf{A}_k^{-1} T \mathbf{A}_k$. That is, each bounded linear operator acting on $\mathcal{H}_{k}$ is uniraty equivalent to a $(k\times k)$-matrix operator with entries from $\mathcal{L}(\mathcal{H}_1)$.

It is a matter of a simple calculation to see that the operators $J^+_k$, $J^0_k$, and $J^=_k$, acting on $\mathcal{H}_{k}$, are unitary equivalent to the following operators, acting on $\mathbb{R}^k \otimes \mathcal{H}_1$,
\begin{equation*}
 \widetilde{J^+_k} = M_k^+ \otimes I, \quad \widetilde{J^0_k} = M_k^0 \otimes I, \quad \widetilde{J^-_k} = M_k^- \otimes I,
\end{equation*}
where the $(k \times k)$-matrices $M^+_k$, $M^0_k$, and $M^-_k$ are given by
\begin{eqnarray*}
 M^+_k &=& -
 \begin{bmatrix}
0 & 0 & 0 & \cdots & 0 & 0 \\
k-1 & 0 & 0 & \cdots & 0 & 0 \\
0 & \sqrt{2}(k-2) & 0 & \cdots & 0 & 0 \\
\vdots  & \vdots & \vdots   &  & \ddots  & \vdots \\
0 & 0 & 0 & \cdots & \sqrt{k-1} & 0
\end{bmatrix}, \\
M^0_k &=&
 \begin{bmatrix}
\textstyle{\frac{1-k}{2}} & 0 & \cdots & 0 \\
0 & \textstyle{\frac{3-k}{2}} &  \cdots  & 0 \\
\vdots  & \vdots    & \ddots  &  \vdots \\
0 & 0 &  \cdots & \textstyle{\frac{k-1}{2}}
\end{bmatrix}, \\
M^-_k &=&
 \begin{bmatrix}
0 & 1 & 0 & \cdots &  0 \\
0 & 0 & \sqrt{2} & \cdots & 0  \\
\vdots  & \vdots & & \ddots  & \vdots \\
0 & 0 & 0 & \cdots & \sqrt{k-1} \\
0 & 0 & 0 & \cdots  & 0
\end{bmatrix}\ ,
\end{eqnarray*}
and form the $(k \times k)$-matrix representation of the generators of the algebra $\mathfrak{sl}(2)$.
Note that the matrices $M^+_k$, $M^0_k$, and $M^-_k$ have no non-trivial common invariant subspace in $\mathbb{R}^k$. Thus they generate the algebra $\mathrm{Mat}_k(\mathbb{R})$, equivalently, each  $(k \times k)$-matrix is a polynomial of $M^+_k$, $M^0_k$, and $M^-_k$. Note that by dimension reasoning there can not be more then $k^2$ linearly independent such polynomials. At the same time there are $k^2$ linearly independent  polynomials $P_{m,n}(M^+_k, M^0_k, M^-_k) = E_{m,n}$ realizing matrix-units of $(k \times k)$-matrix. This leads us to the following proposition.

\begin{proposition} \label{inv-F_k}
 Each bounded linear operator acting on $\mathbb{R}^k \otimes \mathcal{H}_1$ can be uniquely represented in the form
 \begin{equation*}
  \widetilde{T} = \sum_{n,m = 1}^k P_{m,n}(\widetilde{J^+_k}, \widetilde{J^0_k}, \widetilde{J^-_k)}\, S_{m,n},
 \end{equation*}
where $S_{m,n} \in \mathcal{L}(\mathcal{H}_1)$, $m,n =1,2,\ldots,k$. \\
Thus each bounded linear operator acting on the $k$-poly-$\ag$-space $\mathcal{H}_k$ can be uniquely represented in the form
\begin{equation*}
T = \mathbf{A}_k \widetilde{T} \mathbf{A}_k^{-1} =\sum_{n,m = 1}^k P_{m,n}(J^+_k, J^0_k, J^-_k)\, \mathbf{A}_{(m)}S_{m,n}\mathbf{A}_{(n)}^{-1},
\end{equation*}
with $S_{m,n} \in \mathcal{L}(\mathcal{H}_1)$, $m,n =1,2,\ldots,k$.
\end{proposition}

Restricting the class of operators from $\mathcal{L}(\mathcal{H}_1)$ to just scalar operators we come to
\begin{corollary} \label{F_k-invariant}
 Each operator $P(J^0_k, J^-_k)$ being a polynomial of $J^0_k$ and $J^-_k$ with complex coefficients acts invariantly on each $\mathcal{H}_k$, is bounded there,
 and preserves the flag \eqref{flagH}.

  Additionally, each polynomial $P(\ag^\dag \ag)= P(J^0_k+\frac{k-1}{2}I)$ with complex coefficients preserves all true-poly-$\ag$-spaces $\mathcal{H}_{(n)}$ and acts boundedly on each of them.
\end{corollary}

Note that in the standard case of the one-dimensional $\mathcal{H}_1$ (single vacuum vector) the above statements are fundamental in the study of the so-called \emph{exactly-solvable} and  \emph{quasi-exactly-solvable} problems, see \cite{Turbiner:1988,Turbiner:1994,t2,t-sl22}, for review see also \cite{t-2016}.

\section{Poly-analytic functions in one variable revisited}
\label{oneFock}

 We start with $\mathcal{H}$ being the space $L_{2}(\C,d\mu)$ of square-integrable
on $\C$ with the Gaussian measure
\begin{equation*}
   d\mu(z)=\pi^{-1}\,e^{-z\cdot {\bar z}}dv(z)\ ,
\end{equation*}
where $dv(z)=dxdy$ is the Euclidean volume measure on $\C=\R^{2}$, and the following \cite[Section 2]{Vasilev00} lowering and raising operators
\begin{equation}
\label{a^dag-a}
 \ag\ =\ \frac{\pa}{\pa \overline{z}}, \qquad \ag^\dag\ =\ -\frac{\pa}{\pa z} + \overline{z}.
\end{equation}
Recall that the classical Fock \cite{Berezin72,Fock32} (or Segal-Bargmann \cite{Bargmann61,Segal60}) space $F^{2}(\C)$ is the closed subspace of $L_{2}(\C,d\mu)$, which consists of all analytic in $\C$ functions.

Alternatively, it can be defined as the (closed) subspace of all smooth functions satisfying the Cauchy--Riemann equation
\begin{equation}
\label{A}
   \ag\,\varphi = \frac{\pa \varphi}{\pa \overline{z}}\ =\ 0\ ,
\end{equation}
implying that each normalized analytic function can serve as a vacuum vector.

We recall now some results from \cite{Vasilev00}, where all proofs and details can be found.  Alternatively some of them yields from the considerations in the previous section.\\
Besides the Fock space $F^{2}(\C)$, we additionally
introduce the poly-Fock spaces, i.e., for each $k \in \N$, the $k$-Fock (or $k$-poly-Fock)
space $F^2_k(\C)$ ($k$-poly-$\ag$-space, as defined by \eqref{a^k}) is the closed set of all smooth functions from
$L_2(\C,d\mu)$ satisfying the equation
\begin{equation} \label{kA}
\ag^k\, \varphi = \left(\frac{\pa}{\pa\overline{z}}\right)^k\varphi\ =\ 0\ .
\end{equation}
It is convenient to introduce the spaces
\begin{eqnarray*}
  F^2_{(k)}(\C) &=& F^2_{k}(\C) \ominus F^2_{k-1}(\C), \ \ \ \ {\rm
  for}  \ \ k>1, \\
  F^2_{(1)}(\C) &=& F^2_{1}(\C) = F^2(\C), \ \ \ \ {\rm
  for}  \ \ k=1.
\end{eqnarray*}
We call the space $F^2_{(k)}(\C)$ the {\it true-$k$-Fock space}.
It is evident that
\begin{equation} \label{k+p}
  F^2_{k}(\C)\ = \bigoplus_{p=1}^{k} F^2_{(p)}(\C) .
\end{equation}
Then we have

\begin{proposition}\cite[Corollary 2.4]{Vasilev00} \label{co:+}
The space $L_2(\C,d\mu)$ admits the following decomposition
\begin{equation*}
   L_2(\C,d\mu) = \bigoplus_{k=1}^{\infty} F^2_{(k)}(\C) .
\end{equation*}
\end{proposition}

An equivalent to Proposition \ref{co:+} statement can be formulated in terms of the $k$-Fock spaces as follows.
\begin{corollary} \label{co:U}
The set of $k$-Fock subspaces $F^2_k(\C)$, $k\in \N$, of the space $L_2(\C,d\mu)$ forms an infinite flag
in $L_2(\C,d\mu)$
\begin{equation} \label{flag1}
 F^2_1(\C) \ \subset \ F^2_2(\C)\ \subset \ ... \ \subset \ F^2_k(\C) \subset \ ...  \subset \ L_2(\C,d\mu),
\end{equation}
and
\begin{equation*}
   L_2(\C,d\mu) = \overline{\bigcup_{k=1}^{\infty} F^2_{k}(\C}).
\end{equation*}
\end{corollary}

Among other properties of the true-$k$-Fock spaces we mention

\begin{proposition} \cite[Theorem 2.9]{Vasilev00}
 For each $k = 2,3, ...$, the operator
\begin{equation} \label{k-1->k}
  \textstyle{\frac{1}{\sqrt{k-1}}}\,\ag^\dag|_{F_{(k-1)}^2(\mathbb{C})} : F_{(k-1)}^2(\mathbb{C}) \longrightarrow F_{(k)}^2(\mathbb{C})
\end{equation}
is an isometric isomorphism, and the operator
\begin{equation} \label{k->k-1}
   \textstyle{\frac{1}{\sqrt{k-1}}}\,\ag|_{F_{(k)}^2(\mathbb{C})} : F_{(k)}^2(\mathbb{C}) \longrightarrow F_{(k-1)}^2(\mathbb{C})\ ,
\end{equation}
is its inverse.

And for each $k \in \N$, the operator
\begin{equation*} 
  \mathbf{A}_{(k)} := \textstyle{\frac{1}{\sqrt{(k-1)!}}}\, (\ag^\dag)^{\,k-1}|_{F^2(\C)}  : F^2(\C) \ \longrightarrow \ F_{(k)}^2(\C)\ ,
\end{equation*}
gives an isometric isomorphism between the Fock space $F^2(\C)$ and the true-poly-Fock space $F_{(k)}^2(\C)$,
and the operator
\begin{equation*}
  \mathbf{A}_{(k)}^{-1} := \textstyle{\frac{1}{\sqrt{(k-1)!}}}\,{\ag}^{k-1}|_{F_{(k)}^2(\C)} : F_{(k)}^2(\C)\ \longrightarrow  \ F^2(\C)\ ,
\end{equation*}
is an inverse isomorphism.
\end{proposition}
\noindent It is also worth recalling that $F_1^2(\mathbb{C})=F_{(1)}^2(\mathbb{C})=F^2(\mathbb{C})$ is the kernel of the operator $\ag$.

As a direct corollary we have

\begin{corollary} \label{true-k_rep} \cite[Corollary 2.10]{Vasilev00}
Each function $\psi(z,\overline{z})$ from the true-$k$-Fock space $F_{(k)}^2(\C)$ is uniquely defined by a function $\varphi(z) \in F^2(\C)$ and has the form
\begin{equation} \label{eq:F_(k)}
  \psi(z)=\psi(z,\overline{z})=\sum_{m=0}^{k-1} (-1)^m
  \frac{\sqrt{(k-1)!}}{m!\,(k-1-m)!}\,
  \overline{z}^{\,k-1-m}\,\varphi^{(m)}(z),
\end{equation}
where $\varphi^{(m)}$ is the $m$-th derivative
of the (analytic) function $\varphi$, and
\begin{equation*}
  \|\psi\|_{F_{(k)}^2(\C)} = \|\varphi\|_{F^2(\C)}.
\end{equation*}
\end{corollary}

\medskip
The above proposition permits us to characterize explicitly elements of the $k$-Fock space $F_k^2(\C)$.

\begin{theorem} \label{poly<-Fock}
 Each function $\varphi(z,\overline{z}) \in F_k^2(\C)$ is uniquely defined by $k$ functions $f_1(z)$, ..., $f_k(z)$ from the Fock space $F^2(\C)$ and admits the representation
 \begin{equation} \label{by-ell}
   \varphi(z,\overline{z})= \sum_{\ell=1}^{k} \overline{z}^{\ell-1} \cdot
   \varphi_{\ell} (z),
\end{equation}
where the analytic in $\C$ functions $\varphi_{\ell}(z)$ have the form
\begin{equation*}
 \varphi_{\ell}(z) = \sum_{p=\ell}^k (-1)^{p-\ell} \frac{\sqrt{(p-1)!}}{(p-\ell)!\,(\ell-1)!}\, f_p^{(p-\ell)}(z).
\end{equation*}
\end{theorem}

\begin{proof}
By \eqref{k+p}, each function $\varphi(z,\overline{z}) \in F_k^2(\C)$ admits the unique representation
\begin{equation} \label{by-psi}
 \varphi(z,\overline{z})= \sum_{p=1}^k \psi_{(p)}(z,\overline{z}),
\end{equation}
and, by \eqref{eq:F_(k)},
\begin{eqnarray*}
  \psi_{(p)}(z,\overline{z})&=&\sum_{m=0}^{p-1} (-1)^m
  \frac{\sqrt{(p-1)!}}{m!\,(p-1-m)!}\, \overline{z}^{\,p-1-m}\,f_p^{(m)}(z) \\
  &=& \sum_{\ell=1}^{p} (-1)^{p-\ell}
  \frac{\sqrt{(p-1)!}}{(p-\ell)!\,(\ell-1)!}\, \overline{z}^{\,\ell-1}\,f_p^{(p-\ell)}(z),
\end{eqnarray*}
for a certain $f_p(z) \in F^2(\C)$. Substituting now the above expression for $\psi_{(p)}(z,\overline{z})$ to \eqref{by-psi}, and collecting terms with $\overline{z}^{\ell-1}$, we obtain that
\begin{equation*}
   \varphi(z,\overline{z})= \sum_{\ell=1}^{k} \overline{z}^{\ell-1} \cdot
   \varphi_{\ell} (z),
\end{equation*}
where
\begin{equation*}
 \varphi_{\ell}(z) = \sum_{p=\ell}^k (-1)^{p-\ell} \frac{\sqrt{(p-1)!}}{(p-\ell)!\,(\ell-1)!}\, f_p^{(p-\ell)}(z). \qedhere
\end{equation*}
\end{proof}
\noindent Note that $\varphi_k(z)=\frac{f_k(z)}{\sqrt{(k-1)!}} \in F^2(\C)$, while the others $\varphi_{\ell}(z)$, with $\ell = 1,2,\ldots, k-1$, generically do not belong to  $F^2(\C)$.

\begin{observation} {\rm
 For a function $\varphi(z,\overline{z}) \in F_{(k)}^2(\C)$ represented in form \eqref{by-ell}, the following recursive formulas allow us to recover all functions $f_1,\, f_2,\,\ldots,\, f_k$ from $F^2(\C)$ generated it
\begin{eqnarray*}
 f_k(z) &=& \sqrt{(k-1)!}\,\varphi_k(z), \ \ \ldots \ \ , \\
 f_{\ell}(z) &=& \sqrt{(\ell-1)!}\left[\varphi_{\ell}(z) - \sum_{p=\ell+1}^k (-1)^{p-\ell} \frac{\sqrt{(p-1)!}}{(p-\ell)!\,(\ell-1)!}\, f_p^{(p-\ell)}(z) \right], \quad \mathrm{for \ all} \ \ \ell < k.
\end{eqnarray*}
}
\end{observation}

\medskip
Further, each true-$k$-Fock space is a reproducing kernel Hilbert space and
\begin{theorem}\cite[Lemma 3.2 and Theorem 3.3]{Vasilev00}
The operator
\begin{equation} \label{P_(k)}
  (P_{(k)}f)(z)=\langle f(\zeta),\,q_z^{\{k\}}(\zeta)\rangle =
  \int_{\C} f(\zeta)\,q_{\zeta}^{\{k\}}(z)\,d\mu(\zeta)
\end{equation}
is the orthogonal Bargmann projection of $L_2(\C,d\mu)$ onto the
true-$k$-Fock space $F^2_{(k)}(\C)$,
where the reproducing kernel $q_z^{\{k\}}(\zeta)$ is given by
\begin{equation*}
  q_z^{\{k\}}(\zeta) = \frac{1}{(k-1)!} \, \left(- \frac{\pa}{\pa \zeta} +
   \overline{\zeta} \right)^{k-1} \left(- \frac{\pa}{\pa \overline{z}} +
   z \right)^{k-1} e^{\zeta\overline{z}},  \ \ \ \ \zeta \in \C.
\end{equation*}
For each $z \in \C$, the function $q_z^{\{k\}}(\zeta)$ belongs to $F^2_{(k)}(\C)$, has the form
\begin{equation*} \label{eq:q}
  q_z^{\{k\}}(\zeta)=e^{\zeta\overline{z}}
  p_{k-1}((z-\zeta)(\overline{\zeta}-\overline{z})),
\end{equation*}
for certain real coefficient polynomial $p_{k-1}(\lambda)$, and
\begin{equation*}
  q_{\zeta}^{\{k\}}(z) = \overline{q_z^{\{k\}}(\zeta)}.
\end{equation*}
\end{theorem}

\bigskip
In a similar way as in Section \ref{se:overture} (see \eqref{bk-ak}) the condition (\ref{kA}) is equivalent to
\begin{equation*}
 (\ag^\dag)^k\, \ag^k\, \varphi = \left(-\frac{\pa}{\pa z} + \overline{z}\right)^k\left(\frac{\pa}{\pa \overline{z}} \right)^k  \varphi\ =\ 0\
\end{equation*}
with
\begin{equation*}
 {(\ag^\dag)}^k\,\ag^k\, = \, \prod_{m=0}^{k-1}\left(\ag^\dag\, \ag - m\right),
\end{equation*}
where the right hand side is a product of the so-called Euler-Cartan operators
\begin{equation} \label{Cartan}
 I_0^{(m)}\ =\  \ag^\dag \ag - m\ =\  \left(-\frac{\pa}{\pa z} + \overline{z}\right)\,\frac{\pa }{\pa \overline{z}} - m,
   \ \quad m=0,1,\ldots k-1 .
\end{equation}
corresponding to the representation marked by $m$ of the algebra $\mathfrak{sl}(2)$.
Formulas \eqref{k-1->k} and \eqref{k->k-1} imply that $I_0^{(m)}$ boundedly acts on each $F_{(p)}^2$ and has  $F_{(m)}^2$ as its kernel. We note as well that each of the operators
\begin{eqnarray}
\label{sl2ag-ag+}
   J^+_k &=&  (\ag^\dag)^2 \ag -
   (k-1)\, \ag^\dag = \ag^\dag (\ag^\dag \ag - (k-1)\,I)
   = \left(\overline{z}-\textstyle{\frac{\partial}{\partial z}}\right)\left(\overline{z} \textstyle{\frac{\partial}{\partial \overline{z}}} - \frac{\partial^2}{\partial z \partial \overline{z}}- (k-1)I\right), \non
\\
   J^0_k &=&  \ag^\dag \ag - \textstyle{\frac{k-1}{2}}I = \overline{z} \frac{\partial}{\partial \overline{z}} - \frac{\partial^2}{\partial z \partial \overline{z}}- \textstyle{\frac{k-1}{2}}I,
\\
   J^-_k &=& \ag = \textstyle{\frac{\partial}{\partial \overline{z}}}\non .
\end{eqnarray}
is bounded on each poly-Fock space $F_{p}^2$.

It is straightforward that for all positive integers $k$ the above operators  $J^+_k$, $J^0_k$ and $J^-_k$ act invariantly on $k$-Fock space $F_k^2$. Thus, the poly-Fock (poly-analytic) spaces $F^2_k(\C)$, $k=1,2,...$ can be defined alternatively as the maximal (in a sense of Observation \ref{max-inv}) invariant subspaces for the action the operators $J^+_k$,
$J^0_k$, $J^-_k$ in $L_2(\C,d\mu)$, obeying the $\mathfrak{sl}(2)$-algebra commutation relations.

\medskip
We characterize now some invariance properties of the $k$-poly-Fock spaces. \\
As well known any motion of the complex plane $\C$, a one-to-one mapping of $\C$ that preserves distances and does not changes the orientation, is a combination of rotations ($z \mapsto w=\alpha z$, $|\alpha|=1$) and  parallel translations ($z \mapsto w=z+a$, $a \in \C$),
generating the group $E_2$ of motions in $\C$. The corresponding operators, that act unitarily on both $L_2(\C,d\mu)$ and $F^2(\C)$, have the form of rotation $(U_{\alpha}f)(w)= f(\overline{\alpha}\,w)$ and the Weyl operator (see e.g. \cite[Section 2]{Zhu-Fock}) $(W_af)(w) = e^{\overline{a}w-\frac{1}{2}|a|^2} f(w-a)$, being a shift operator with a gauge factor.

\begin{theorem} \label{inv-C}
 For each $k$, the true $k$-Fock space $F^2_{(k)}(\C)$ is invariant under the action of the operators $U_{\alpha}$, $|\alpha|=1$ and $W_a$, $a \in \C$, where they act isometrically.
\end{theorem}

\begin{proof}
Me mention first that a routine verification gives
\begin{equation*}
 U_{\alpha} \ag^\dag U_{\alpha}^{-1} = \alpha\, \ag^\dag \qquad \mathrm{and} \qquad W_a \ag^\dag W_a^{-1} = \ag^\dag
\end{equation*}
We proceed then by induction on $k$. For $k=1$, i.e. for the Fock space, the result is already known. Assume that for each $k-1$, $k=2,3,\ldots$, the result is valid: for each $\psi \in F^2_{(k-1)}(\C)$
\begin{equation*}
 U_{\alpha}\psi = \psi_{\alpha} \in F^2_{(k-1)}(\C) \qquad \mathrm{and} \qquad
 W_a \psi = \psi_a \in F^2_{(k-1)}(\C).
\end{equation*}
Then, using \eqref{k-1->k} and assuming that $\varphi = \frac{1}{\sqrt{k-1}}\ag^\dag \psi \in F^2_{(k)}(\C)$,
\begin{eqnarray*}
 U_{\alpha} \varphi &=& U_{\alpha} \frac{1}{\sqrt{k-1}}\ag^\dag U_{\alpha}^{-1}\left(U_{\alpha}\psi \right) = \alpha\,\frac{1}{\sqrt{k-1}}\ag^\dag \psi_{\alpha} \, \in \, F^2_{(k)}(\C), \\
 W_a \psi &=& W_a \frac{1}{\sqrt{k-1}}\ag^\dag W_a^{-1} \left(W_a\psi \right) =  \frac{1}{\sqrt{k-1}}\ag^\dag \psi_a \, \in \, F^2_{(k)}(\C).
\end{eqnarray*}
The isometric action of $U_{\alpha}$ and $W_a$ follows from the isometric action of $\frac{1}{\sqrt{k-1}}\ag^\dag$ on each $F^2_{(k-1)}(\C)$, for all $k=2,3,\ldots$, and $|\alpha|=1$.
\end{proof}

\begin{corollary}
 The unitary operators associated with motions of the plane $\C$ preserve flag \eqref{flag1} of $k$-Fock spaces.
\end{corollary}

 The general description of operators boundedly acting and preserving $k$-poly-Fock spaces $F^2_k(\C)$, flag \ref{flag1} of these spaces, or all true-$k$-poly-Fock spaces $F^2_{(k)}(\C)$ is given by Proposition \ref{inv-F_k}  and Corollary~\ref{F_k-invariant}.

Note that among the operators acting on each $F^2_{(k)}(\C)$ there is well known and important for physics applications operator.

\begin{example} {\sf Landau magnetic Hamiltonian.}

{\rm Let us consider the operator
\begin{equation}
\label{tDelta}
 \widetilde{\Delta}\ =\ - \frac{\pa^2}{\pa z\,\pa\overline{z}} + \overline{z}\,\frac{\pa}{\pa\overline{z}}\ ,
\end{equation}
which (in suitable units and up to an additive constant) is a realization on $L_{2}(\C,d\mu)$ of the similarity-transformed Schr\"odinger operator (which we call here the {\it Landau magnetic Hamiltonian}) describing the transverse motion of a charged particle evolving in the complex plane $\mathbb{C}$ subject to a normal uniform constant magnetic field in asymmetric (Landau) gauge, see e.g. \cite{LL:1977}, Chapter XV, $\S$112. The first term in (\ref{tDelta}) has a meaning of the kinetic energy.

It is immediately recognized that (\ref{tDelta}) is the Euler-Cartan generator of the algebra $\mathfrak{sl}(2)$ (\ref{sl2ag-ag+}) at $k=1$ and can be rewritten as follows
\begin{equation*}
\label{aa}
  \widetilde{\Delta}\ =\ J^0_1\ =\ \ag^\dag \ag\ ,
\end{equation*}
with the operators $\ag^\dag$ and $\ag$ given by \eqref{a^dag-a}. By Corollary \ref{F_k-invariant} the operator $\widetilde{\Delta}$ acts invariantly on each $k$-poly-Fock space $F_{k}^2(\mathbb{C})$ and preserves the infinite flag of poly-Fock spaces.  Furthermore, by \eqref{k-1->k} and \eqref{k->k-1}, each true-poly-Fock space $F_{(k)}^2(\mathbb{C})$ is invariant subspace for the operator $\widetilde{\Delta} = \ag^\dag \ag$ as well and
\begin{equation*}
 \widetilde{\Delta}|_{F_{(k)}^2(\mathbb{C})} = (k-1) I\ , \quad\ k \in \mathbb{N}.
\end{equation*}
Hence, the spectrum of the operator $\widetilde{\Delta}$ consists of infinitely many equidistant eigenvalues, each of infinite multiplicity (Landau levels), they are of the form
\begin{equation*}
 \lambda_k = k-1\ , \qquad k \in \mathbb{N}\ ,
\end{equation*}
and the corresponding eigenspaces are nothing but the true-$k$-poly-Fock spaces $F_{(k)}^2(\mathbb{C})$.
}
\end{example}

Let us mention an alternative way to get the same result. Again, by Corollary \ref{F_k-invariant} the operator $\widetilde{\Delta}$ acts invariantly on each true-$k$-poly-Fock space $F_{(k)}^2(\mathbb{C})$, which consists of all functions $\psi = (\ag^\dag)^{k-1}f$, with $f \in F^2(\C)$. Consider then the corresponding spectral problem on $F_{(k)}^2(\mathbb{C})$
\begin{equation*}
 \widetilde{\Delta}\,\psi = \lambda \, \psi \qquad \mathrm{or} \qquad (\ag^\dag \ag)(\ag^\dag)^{k-1}f = \lambda \, (\ag^\dag)^{k-1}f.
\end{equation*}
Making use of \eqref{[a,b]} and $\ag f =0$ we arrive at
\begin{equation*}
 (k-1)(\ag^\dag)^{k-1}f = \lambda \, (\ag^\dag)^{k-1}f \qquad \mathrm{or} \qquad
 \widetilde{\Delta}\,\psi = (k-1)\psi, \quad \mathrm{for \ all} \ \ \psi \in F_{(k)}^2(\mathbb{C}).
\end{equation*}

\begin{observation}
 {\rm This leads us to the following approach to the spectral problems for the operators $\mathcal{D} = P(J_k^+,J_k^0,J_k^-)$, with $P$ being a polynomial with complex coefficients. By Proposition~\ref{inv-F_k} the operator $\mathcal{D}$ acts invariantly on the $k$-poly-Fock space $F_{k}^2(\mathbb{C})$, which implies the following form of the spectral problem for $\mathcal{D}$ restricted to $F_{k}^2(\mathbb{C})$
 \begin{equation*}
  \mathcal{D} \psi = \lambda \psi, \quad \mathrm{where} \quad \psi \in F_{k}^2(\mathbb{C}),
 \end{equation*}
or
\begin{equation*}
 \left(P(J_k^+,J_k^0,J_k^-) - \lambda\right)\sum_{j=0}^{k-1} (\ag^\dag)^j f_j = 0, \quad \mathrm{where} \ \ f_j \in F^2(\C), \ \ \mathrm{for \ all} \ \ j =0,1,\ldots, k-1.
\end{equation*}
Using \eqref{[a,b]} and $\ker \ag = F^2(\C)$, we rearrange then the operator on the left-hand side to its so-called Wick normal form (powers of $\ag^\dag$ to the left and powers of $\ag$ to the right). This results that a certain $k$-poly-analytic function, being the sum of its mutually orthogonal true-$(j+1)$-poly-analytic functions, $j =0,1,\ldots, k-1$, is identically zero. Setting these true-poly-analytic components to zero we come to a system of $k$ equations for $f_j \in F^2(\C)$ which can be solved then by the linear algebra means.
 }
\end{observation}

We illustrate this on a modified Landau Hamiltonian, where we change the "mark" $1$ in $J^0_1$  to  mark $2$ (for simplicity) with extra constant term and by adding the generators $J^{\pm}_2$ (\ref{sl2ag-ag+}),

\begin{example}
   {\rm
Let us introduce the operator
\begin{equation*}\label{tDelta-sl2-k}
 \widetilde{\Delta_2} = J^0_2  + \textstyle{\frac{1}{2}} \,I\ +\ \alpha J^+_2 \ +\ \beta J^-_2 = \ag^\dag \ag\ +\ \alpha\,\ag^\dag \left(\ag^\dag \ag - I \right) + \beta \ag,
\end{equation*}
with parameters $\alpha$ and $\beta$. The operator $\widetilde{\Delta_2}$ is a differential operator of the third order which becomes of the second order if $\alpha = 0$. By Proposition \ref{inv-F_k} the $2$-poly-Fock space $F_{2}^2(\mathbb{C})$ is its invariant subspace, while if $\alpha = 0$, by Corollary \ref{F_k-invariant}, $\widetilde{\Delta_2}$ acts invariantly on each poly-Fock space preserving the flag \eqref{flag1}.

The corresponding spectral problem has the form
\begin{equation*}
 \widetilde{\Delta_2} \psi = \lambda \psi, \quad \mathrm{where} \ \ \psi \in F_{2}^2(\mathbb{C}),  \quad \mathrm{or} \quad \left(\ag^\dag \ag\ + \alpha\,\ag^\dag (\ag^\dag \ag - I) + \beta \ag - \lambda\right)(\ag^\dag f_1 + f_0) = 0.
\end{equation*}
Let first $\alpha \neq 0$. Then rearranging to the Wick normal form and making all cancellations, we have
\begin{equation*}
 \ag^\dag((1-\lambda)f_1- \alpha f_0) + (\beta f_1 - \lambda f_0) = 0.
\end{equation*}
Equations
\begin{equation*}
 (1-\lambda)f_1- \alpha f_0 = 0 \quad \mathrm{and} \quad \beta f_1 - \lambda f_0 = 0
\end{equation*}
imply that $f_1$ and $f_0$ are related as $\alpha f_0 = \lambda(1-\lambda)f_1$ and the following formula for eigenvalues
\begin{equation*}
 \lambda_{1,2} = \frac{1}{2} (1 \pm \sqrt{1- 4\alpha \beta}).
\end{equation*}
Finally, the corresponding infinite dimensional eigenspaces consist of all $F_{2}^2(\mathbb{C})$-functions of the form
\begin{equation*}
 \psi_{1,2} = \left( \ag^\dag + \frac{1}{\alpha}(1 - \lambda_{1,2})\right)f, \quad f \in  F^2(\C).
\end{equation*}

Let now $\alpha = 0$, then
\begin{equation} \label{alpha=0}
 \widetilde{\Delta_2} = J^0_k  + \textstyle{\frac{1}{2}} \,I\ +\ \beta J^-_2 = \ag^\dag \ag\ + \beta \ag.
\end{equation}
It is clear that the second order operator $\widetilde{\Delta_2}$ is isospectral to $\widetilde{\Delta}$: the spectra remains unchanged, at the same time the corresponding eigenspaces are quite different.
A straightforward calculation shows that the infinite dimensional eigenspace that corresponds to the eigenvalue $\lambda_k = k-1$ consists of all $F_{k}^2(\mathbb{C})$-functions of the form
\begin{equation*}
 \psi = \left(\ag^\dag + \beta \right)^{k-1} f, \quad f \in F^2(\C).
\end{equation*}
That is, the operator \eqref{alpha=0} acts on each $k$-poly-Fock space $F_{k}^2(\mathbb{C})$ invariantly, has there $k$ eigenvalues $\lambda_j=j-1$, $j=1,2,\ldots,k$, whose corresponding eigenspaces are of the form
\begin{equation*}
 \left(\ag^\dag + \beta \right)^{j-1} f, \quad \mathrm{for \ all} \ \ f \in F^2(\C), \quad j=1,2,\ldots,k.
\end{equation*}
}
\end{example}

\section{Poly-analytic functions of several complex variables}
\label{multiFock}

Consider now spaces over $\C^d$ and introduce the spaces
\begin{eqnarray}
  L_2(\C^d,d\mu_d) &=& \underbrace{L_2(\C,d\mu) \otimes L_2(\C,d\mu)\ldots
 \otimes L_2(\C,d\mu)}_{d \ \mathrm{times}}\ ,\non
  \\
  F^2(\C^d) &=& \underbrace{F^2(\C) \otimes F^2(\C)\ldots  \otimes
  F^2(\C)}_{d \  \mathrm{times}}\ ,
\end{eqnarray}
where the Fock space $F^2(\C^d)$ is the closed subspace of $L_2(\C^d,d\mu_d)$, which consists of analytic functions in $d$ complex variables.
Given a multi-index $\mathbf{k}=(k_1,k_2,\ldots k_d)$, introduce the true-$\mathbf{k}$-Fock space $F^2_{(\mathbf{k})}(\C^d)$ as the
following tensor product of  the true-poly-Fock spaces over~$\C$
\begin{equation*} 
  F^2_{(\mathbf{k})}(\C^d) = \underbrace{F^2_{(k_1)}(\C) \otimes
  \ldots \otimes F^2_{(k_d)}(\C)}_{d\  \mathrm{times}}\, .
\end{equation*}
In particular, for the multi-index $\mathbf{1}=(1,1,\ldots 1)$, we have
\begin{equation*}
  F^2_{\mathbf{1}}(\C^d) = F^2_{(\mathbf{1})}(\C^d) = F^2(\C^d) .
\end{equation*}

The above implies that the orthogonal projection $P_{(\mathbf{k})} : \, L_2(\C^d,d\mu_d) \rightarrow F^2_{(\mathbf{k})}(\C^d)$ has the form
\begin{equation} \label{P_(bf-k)}
 P_{(\mathbf{k})} = P_{k_1} \otimes \ldots \otimes P_{k_d},
\end{equation}
where each $P_{k_j}$ is given by \eqref{P_(k)}.

\begin{theorem}[\cite{Vasilev00}] \label{th:+}
The space $L_2(\C^d,d\mu_d)$ admits the following decomposition
\begin{equation*}
   L_2(\C^d,d\mu_d) = \bigoplus_{|\mathbf{k}|=1}^{\infty}
   F^2_{(\mathbf{k})}(\C^d)\ .
\end{equation*}
\end{theorem}

Introduce the $(2d+1)$-dimensional Heisenberg algebra $\mathbb{H}_{2d+1} = \{ \ag,\ag^\dag,1 \}$ with commutator $[\ag_{i}, \ag^\dag_j]\ = \ \delta_{ij}\,I,\ i,j = 1,2,\ldots,d\ ,\ [\ag_{i}, \ag_j]=[\ag^\dag_{i}, \ag^\dag_j]=0$ and $[\ag_{i}, 1]=[\ag^\dag_{i}, 1]=0$. Its representation on $L_2(\C^d,d\mu_d)$ is given by
$d$ pairs of raising and lowering operators related to different $z_j$ in $z=(z_1,z_2,\ldots,z_d) \in \C^d$:
\begin{equation*}
 \ag^\dag_j = {\bar z}_j -\frac{\pa}{\pa {z}_j}, \qquad \ag_j = \frac{\pa}{\pa {\bar z}_j},
\end{equation*}
the identity operator $I$, with $[\ag_j,\ag^\dag_j]=I$, $j = 1,2,\ldots,d$, and the closed infinite dimensional space of vacuum vectors $\mathcal{H}_1$ which is given by
\begin{equation*}
 \mathcal{H}_1 = \bigcap_{j=1}^d \mathrm{ker}\, \ag_j = F^2(\C^d)\ .
\end{equation*}
For each multi-index $\mathbf{k}=(k_1,k_2,\ldots k_d)$ the true-$\mathbf{k}$-Fock space $F^2_{(\mathbf{k})}(\C^d)$ again is isomorphic to the space of vacuum vectors $\mathcal{H}_1 = F^2(\C^d)$ and is the result of its ``lifting up'' by the normalised product of the raising operators $(\ag^\dag_1)^{k_1-1}(\ag^\dag_2)^{k_2-1}\cdots (\ag^\dag_d)^{k_d-1}$. Corollary \ref{true-k_rep} implies now
\begin{lemma}
 Given a multi-index $\mathbf{k}=(k_1,k_2,\ldots k_d)$, each function $\psi(z,\overline{z})$ from the true-$\mathbf{k}$-Fock space $F_{(\mathbf{k})}^2(\C^d)$ is uniquely defined by a function $\varphi(z) \in F^2(\C^d)$ and has the form
\begin{equation} \label{eq:F_(k)-d}
  \psi(z)=\psi(z,\overline{z})=\left(\prod_{p=1}^d\sum_{m_p=0}^{k_p-1} (-1)^{m_p}
  \frac{\sqrt{(k_p-1)!}}{m_p!\,(k_p-1-m_p)!}\,
  \overline{z_p}^{\,k_p-1-m_p} \frac{\partial^{m_p}}{\partial z_p^{m_p}}\right)\varphi(z),
\end{equation}
moreover
\begin{equation*}
  \|\psi\|_{F_{(\mathbf{k})}^2(\C^d)} = \|\varphi\|_{F^2(\C^d)}.
\end{equation*}
\end{lemma}

True-$\mathbf{k}$-Fock spaces, being the ``elementary pieces'' in the construction of the forthcoming poly-Fock spaces,  can be obviously rearranged (in
infinitely many ways) to various sets of poly-analytic spaces, so that the closure of their union will give $L_2(\C^d,d\mu_d)$. At this stage the question on the most appropriate such rearrangements of the true-$\mathbf{k}$-Fock spaces naturally appears.

Following the pattern of the one-dimensional situation of the previous section, it is quite natural to define each ``distinguished'' set of poly-Fock spaces in a Lie-algebraic way, i.e., as the infinite system of the representation spaces for the action of generators of a certain Lie algebra $\mathfrak{g}$.

\subsection{Primary case: homogeneous poly-analytic functions} \label{homog}

In this case the algebra $\mathfrak{sl}(d+1)$ plays a role of the algebra $\mathfrak{g}$. The simplest (symmetric) representation of the
$\mathfrak{sl}(d+1)$-algebra given by the following combination of the raising and lowering operators $\ag^\dag_{j,i}$ and $\ag_{j,i}$, $i =1,2,\ldots,d$, see~\cite{t-sl22},
\begin{eqnarray} \label{sl(d+1)}
 J_i^- &=& \ag_i = \frac{\pa}{\pa {\bar z}_i}\ ,\qquad \quad
i=1,2,\ldots , d\
, \non\\
 J_{i,j}^0 &=& \ag^\dag_i \ag_j = \left({\bar
z}_i -\frac{\pa}{\pa {z}_i}\right)\frac{\pa}{\pa {\bar z}_j}\ , \qquad i,j=1,2,\ldots , d \
,
 \\
J_i^+ &=&  \left({\bar
z}_i -\frac{\pa}{\pa {z}_i}\right) \left(\sum_{j=1}^{d}\left({\bar
z}_j -\frac{\pa}{\pa {z}_j}\right)\frac{\pa}{\pa {\bar z}_j}-kI\right), \quad
i=1,2,\ldots , d \ . \non
\end{eqnarray}

For any real $k$ these operators obey $\mathfrak{sl}(d+1)$-algebra commutation relations.
Restricting $k$ to positive integers, we define the $k$-homogeneous-Fock space $F^2_{k\textup{-}hom}(\C^d)$ as the (closed) subspace of the smooth functions in $L_2(\C^d,d\mu_d)$ satisfying the equation
\begin{equation} \label{J(m)}
 \prod_{m=0}^{k-1} \, \left(\sum_{i=1}^{d}\left({\bar z}_{i} -\frac{\pa}{\pa {z}_{i}}\right)\frac{\pa}{\pa {\bar z}_{i}}-mI \right)\, f \ =\ 0.
\end{equation}
As in section 3, we see that alternatively  $F^2_{k,hom}(\C^d)$ is the maximal (in a sense of Observation \ref{max-inv}) closed subspace of $L_2(\C^d,d\mu_d)$ being invariant under the action of the operators $J^+_i$, $J^0_{i,j}$, $J^-_i$, $i,j =1,2,...,d$ obeying $\mathfrak{sl}(d+1)$-algebra commutation relations.

As it easily seen,
\begin{equation} \label{poly-summ}
   F^2_{k,hom}(\C^d) =  \left\{\bigoplus  F^2_{(\mathbf{p})}(\C^d)\, : \ \mathbf{p} =(p_1,\ldots,p_d) \in \mathbb{N}^d \ \ \mathrm{with} \ \ |\mathbf{p}| \leq k  \right\}.
\end{equation}
Then each function $\varphi(z,\overline{z}) \in F^2_{k,hom}(\C^d)$ is uniquely defined by   $C^{k-1}_{d+k-1} = \frac{(d_k-1)!}{d!(k-1)!}$ functions from $F^2(\C^d)$ and admits the representation
\begin{equation*}
 \varphi(z,\overline{z}) = \sum_{|\mathbf{m}|=0}^{k-1} \overline{z}^{\mathbf{m}} f_{\mathbf{m}}(z_1,...,z_d),
\end{equation*}
where all $f_{\mathbf{m}}(z_1,...,z_d)$ are analytic functions on $z_1,\, ...,\, z_d$ and
their explicit form in terms of $f_{\mathbf{m}}$ can be given using \eqref{eq:F_(k)-d}.

Furthermore there is an isometric isomorphism
\begin{equation*}
 F^2_{k,hom}(\C^d) \ \cong \ \mathbb{R}^{C^{k-1}_{d+k-1}}\otimes F^2(\C^d),
\end{equation*}
and the orthogonal projection $P_{k,hom}: \, L_2(\C^d,d\mu_d) \rightarrow F^2_{k,hom}(\C^d)$ has the form
\begin{equation*}
 P_{k,hom} = \left\{\bigoplus  P_{(\mathbf{p})}\, : \ \mathbf{p} =(p_1,\ldots,p_d) \in \mathbb{N}^d \ \ \mathrm{with} \ \ |\mathbf{p}| \leq k  \right\} ,
\end{equation*}
where each $P_{(\mathbf{p})}$ is given by \eqref{P_(bf-k)}.

Note that alternatively the $k$-homogeneous-Fock space $F^2_{k\textup{-}hom}(\C^d)$ can be defined as the (closed) subspace of all smooth functions in $L_2(\C^d,d\mu_d)$ satisfying the equations
\begin{equation*}
 \ag_1^{p_1}\cdots \ag_d^{p_d} f =\frac{\partial^{|\mathbf{p}|}}{\partial \overline{z_1}^{p_1}\cdots \partial \overline{z_d}^{p_d}} f = 0, \quad \mathrm{for \ all} \quad |\mathbf{p}| = k.
\end{equation*}

An analog of Corollary \ref{co:U} reads now as follows

\begin{corollary}
The set of $k$-homogeneous-Fock subspaces $F^2_{k,hom}(\C^d)$, $k\in \N$, of the space $L_2(\C^d,d\mu_d)$ forms an infinite flag in $L_2(\C^d,d\mu_d)$
\begin{equation} \label{flagd}
 F^2_{1,hom}(\C^d) \ \subset \ F^2_{2,hom}(\C^d)\ \subset \ ... \ \subset \ F^2_{k,hom}(\C^d) \subset \ ...  \subset \ L_2(\C^d,d\mu_d),
\end{equation}
and
\begin{equation*}
   L_2(\C^d,d\mu) = \overline{\bigcup_{k=1}^{\infty} F^2_{k,hom}(\C^d)}.
\end{equation*}
\end{corollary}

Let us describe now the invariance properties of the $k$-homogeneous-Fock spaces $F^2_{k,hom}(\C^d)$.
\\
Recall that any motion of $\C^d$ is a combination of complex rotations ($z \mapsto w=\mathbf{u} z$, where $\mathbf{u} \in U(d)$ is a unitary matrix) and  parallel translations ($z \mapsto w=z+\mathbf{a}$, $\mathbf{a}= (a_1,\ldots,a_d) \in \C^d$). The corresponding operators, that act unitarily both on $L_2(\C^d,d\mu_d)$ and $F^2(\C^d)$, have the form $(U_{\mathbf{u}}f)(w)= f(\mathbf{u}^{-1}\,w)$ and the Weyl operator $(W_\mathbf{a}f)(w) = e^{\mathbf{a}\cdot w-\frac{1}{2}|\mathbf{a}|^2} f(w-\mathbf{a})$.

\begin{theorem}
 The operators $U_{\mathbf{u}}$, $\mathbf{u} \in U(d)$, and $W_\mathbf{a}$, $\mathbf{a} \in \C^d$ act unitarily on each $k$-homogeneous-Fock space $F^2_{k,hom}(\C^d)$.
\end{theorem}

\begin{proof}
 It is sufficient to check that the operators $U_{\mathbf{u}}$ and $W_\mathbf{a}$ preserve the spaces $F^2_{k,hom}(\C^d)$.  For the operator $U_{\mathbf{u}}$ it follows from the easily verified connection between the operators $J^{(m)}_0$, $m=0,1,\ldots,k-1$, defining $F^2_{k,hom}(\C^d)$ (see \eqref{J(m)}):
 \begin{equation*}
 U_{\mathbf{u}}\left(\sum_{i=1}^{d}\left({\bar
z}_{i} -\frac{\pa}{\pa {z}_{i}}\right)\frac{\pa}{\pa {\bar z}_{i}}-mI\right) U_{\mathbf{u}}^{-1} = \sum_{i=1}^{d} \left({\bar
w}_{i} -\frac{\pa}{\pa {w}_{i}}\right)\frac{\pa}{\pa {\bar w}_{i}}-m I .
\end{equation*}

For the operator $W_\mathbf{a}$ it follows from
\begin{equation*}
 W_\mathbf{a} = W_{a_1} \otimes \cdots \otimes W_{a_d},
\end{equation*}
equality \eqref{poly-summ} and Theorem \ref{inv-C}.
\end{proof}

That is, the unitary operators associated with a motion of the space $\C^d$ preserve flag \eqref{flagd} of $k$-homogeneous-Fock spaces $F^2_{k,hom}(\C^d)$. Moreover
the $k$-homogeneous-Fock spaces $F^2_{k,hom}(\C^d)$ do not depend on the chose of the Cartesian coordinates in $\C^d$. This justifies the adjective \emph{homogeneous} in their definition.

\subsection{$\mathbf{m}$-quasi-homogeneous-poly-analytic functions} \label{quasi-homog}

We start with with a tuple
$$\mathbf{m} = (m_1,m_2,\ldots,m_n)$$
of natural numbers whose sum is equal to $d$, $m_1 +m_2+\ldots +m_n=d$.
Among all possible tuples $\mathbf{m}$ there are two extreme cases:  $\mathbf{m} = (1,1,\ldots,1)$ with $n=d$, and $\mathbf{m} = (d)$ with $n=1$.

For a given tuple $\mathbf{m}$ the complex space $\C^{d}$ can be written as
\[
   \C^d\ =\ \C^{m_1} \oplus \C^{m_2}\oplus \ldots  \oplus \C^{m_n}\ ,
\]
whose points $z=(z_1,\ldots,z_d) \in \C^d$ are arranged in $n$ groups
\begin{equation*}
 z\ =\ (z_{(1)},z_{(2)},\ldots,z_{(n)}),\quad \mathrm{where} \quad z_{(j)}\ =\  (z_{j,1},\ldots ,z_{j,m_j}) \in \C^{m_j} \ .
\end{equation*}
We will use the same arrangement for
\begin{equation*}
 \mathbb{\ag}\ =\ (\ag_1, \ag_2,\ldots,\ag_d) = (\ag_{(1)},\ag_{(2)},\ldots,\ag_{(n)}) \quad \mathrm{and} \quad \ag^\dag = (\ag^\dag _1, \ag^\dag _2,\ldots,\ag^\dag _d) = (\ag^\dag_{(1)},\ag^\dag_{(2)},\ldots,\ag^\dag_{(n)})\ .
\end{equation*}

Correspondingly, the $d$-dimensional Fock space $F^2(\C^d)$ is decomposed as
\begin{equation*}
 F^2(\C^d)\ =\ F^2(\C^{m_1}) \otimes F^2(\C^{m_2})\otimes \ldots  \otimes F^2(\C^{m_n})\ .
\end{equation*}

Then for each tuple as $\mathbf{k}=(k_1,\ldots,k_n)$, where $k_j$'s are positive integers, we define $\mathbf{k}$-$\mathbf{m}$-quasi-homogeneous-Fock space $F^2_{\mathbf{k}\textup{-}\mathbf{m}\textup{-}q\textup{-}hom}(\C^d)$ as the set of all functions satisfying the equations
\begin{equation} \label{k-m-poly}
 \prod_{\ell=0}^{k_j-1}\left(\sum_{i=1}^{m_j}\ag^\dag_{j,i}\ag_{j,i} - \ell I\right)f\ =\ 0, \quad j=1,2,\ldots,n\ .
\end{equation}
The spaces thus defined are invariant for the action of the algebra
\begin{equation*}
 \mathfrak{g} = \mathfrak{sl}(m_1+1) \otimes  \ldots  \otimes \mathfrak{sl}(m_n+1),
\end{equation*}
where the representation of each $\mathfrak{sl}(m_j+1)$, $j=1,2,\ldots, n$, is given by the following combination of $\ag^\dag_{j,i}$ and $\ag_{j,i}$, $i =1,2,\ldots,m_j$
\begin{eqnarray*}
 J_{j,i}^- &=& \ag_{j,i} = \frac{\pa}{\pa {\bar z}_{j,i}}\ ,\qquad \quad
i=1,2,\ldots , m_j\
, \non\\
 J_{j,i,\ell}^0 &=&  \left({\bar
z}_{j,i} -\frac{\pa}{\pa {z}_{j,i}}\right)\frac{\pa}{\pa {\bar z}_{j,\ell}}\ , \qquad i,\ell=1,2,\ldots , m_j \
, \non\\
 J_{j,i}^+ &=& \left({\bar
z}_{j,i} -\frac{\pa}{\pa {z}_{j,i}}\right) \left(\sum_{\ell=1}^{m_j}\left({\bar
z}_{j,\ell} -\frac{\pa}{\pa {z}_{j,\ell}}\right)\frac{\pa}{\pa {\bar z}_{j,\ell}}-k_jI\right), \quad
i=1,2,\ldots , m_j \ . \non
\end{eqnarray*}

For the case $\mathbf{m} = (1,1,\ldots,1)$ the algebra $\mathfrak{g}\ =\ \mathfrak{sl}(2)^{\otimes d}$ occurs, while for the another extreme case $\mathbf{m}= (d)$ the corresponding algebra is $\mathfrak{g}\ =\ \mathfrak{sl}(d+1)$, and this is the case considered in Subsection~4.1. Each intermediate case, defined by a tuple $\mathbf{m} = (m_1,m_2,\ldots,m_n)$, corresponds to the algebra $\mathfrak{g} = \mathfrak{sl}(m_1+1) \otimes  \ldots  \otimes \mathfrak{sl}(m_n+1)$.

Alternative to \eqref{k-m-poly} equations that define the $\mathbf{k}$-$\mathbf{m}$-quasi-homogeneous-Fock space $F^2_{\mathbf{k}\textup{-}\mathbf{m}\textup{-}q\textup{-}hom}(\C^d)$ are as follows
\begin{equation*}
 \ag_{j,1}^{p_1}\cdots \ag_{j,m_j}^{p_{k_j}} f =\frac{\partial^{|\mathbf{p}|}}{\partial \overline{z}_{j,1}^{p_1}\cdots \partial \overline{z}_{j,m_j}^{p_{k_j}}} f = 0, \quad \mathrm{for \ all} \quad |\mathbf{p}| = k_j, \quad j=1,2,\ldots,n.
\end{equation*}

Note that the classes of $\mathbf{m}$-poly-analytic spaces for both extreme cases were defined in different contexts in \cite{Avanissian,Balk} for $\mathbf{m} = (1,1,\ldots,1)$,  and in \cite{Egor,Youssfi} for $\mathbf{m}= (d)$\,, respectively.

\section*{Acknowledgments}

\bigskip

A.V. Turbiner was supported in part by DGAPA grant {\bf IN113819}~(Mexico). N.L. Vasilevski thanks CONACyT grants 280732 and FORDECYT-PRONACES/61517/2020 (Mexico) for partial support.

\end{document}